\pdfoutput=1
\RequirePackage{ifpdf}
\ifpdf 
\documentclass[pdftex]{sigma}
\else
\documentclass{sigma}
\fi

\numberwithin{equation}{section}

\newtheorem{Theorem}{Theorem}[section]
\newtheorem{Lemma}[Theorem]{Lemma}
\newtheorem{Proposition}[Theorem]{Proposition}
 { \theoremstyle{definition}
\newtheorem{Definition}[Theorem]{Definition}
\newtheorem{Example}[Theorem]{Example}
\newtheorem{Remark}[Theorem]{Remark} }

\begin{document}

\allowdisplaybreaks

\newcommand{\arXivNumber}{1605.07010}

\renewcommand{\PaperNumber}{111}

\FirstPageHeading

\ShortArticleName{Hypergroups Related to a Pair of Compact Hypergroups}

\ArticleName{Hypergroups Related to a Pair\\ of Compact Hypergroups}

\Author{Herbert HEYER~$^{\dag^1}$, Satoshi KAWAKAMI~$^{\dag^2}$, Tatsuya TSURII~$^{\dag^3}$ and Satoe YAMANAKA~$^{\dag^4}$}

\AuthorNameForHeading{H.~Heyer, S.~Kawakami, T.~Tsurii and S.~Yamanaka}

\Address{$^{\dag^1}$~Universit\"{a}t T\"{u}bingen, Mathematisches Institut,\\
\hphantom{$^{\dag^1}$}~Auf der Morgenstelle 10, 72076, T\"{u}bingen, Germany}
\EmailDD{\href{mailto:herbert.heyer@uni-tuebingen.de}{herbert.heyer@uni-tuebingen.de}}

\Address{$^{\dag^2}$~Nara University of Education, Department of Mathematics,\\
\hphantom{$^{\dag^2}$}~Takabatake-cho Nara, 630-8528, Japan}
\EmailDD{\href{mailto:kawakami@nara-edu.ac.jp}{kawakami@nara-edu.ac.jp}}

\Address{$^{\dag^3}$~Osaka Prefecture University, 1-1 Gakuen-cho, Nakaku, Sakai Osaka, 599-8531, Japan}
\EmailDD{\href{mailto:dw301003@edu.osakafu-u.ac.jp}{dw301003@edu.osakafu-u.ac.jp}}

\Address{$^{\dag^4}$~Nara Women's University, Faculty of Science, Kitauoya-higashimachi, Nara, 630-8506, Japan}
\EmailDD{\href{mailto:s.yamanaka516@gmail.com}{s.yamanaka516@gmail.com}}

\ArticleDates{Received June 02, 2016, in f\/inal form November 10, 2016; Published online November 18, 2016}

\Abstract{The purpose of the present paper is to investigate a hypergroup associated with irreducible characters of a compact hypergroup~$H$ and a closed subhypergroup~$H_0$ of~$H$ with $ |H/H_0|< + \infty$. The convolution of this hypergroup is introduced by inducing irreducible characters of~$H_0$ to~$H$ and by restricting irreducible characters of~$H$ to~$H_0$. The method of proof relies on the notion of an induced character and an admissible hypergroup pair.}


\Classification{22D30; 22F50; 20N20; 43A62}

\section{Introduction}
The aim of the present paper is to contribute to the largely open problem of establishing a~structure theory of hypergroups. Hypergroups are locally compact spaces on which the bounded measures convolve similar to the group case. The origin of the notion of hypergroup or gene\-ra\-li\-zed translation structure goes back to J.~Delsarte and B.M.~Levitan, the special class of double coset hypergroups appears already in the work of G.~Frobenius.

There exists an axiomatic approach to hypergroups initiated by Charles F.~Dunkl \cite{D1,D2}, R.I.~Jewett~\cite{J}, and R.~Spector~\cite{S}, which lead to an extensive harmonic analysis of hypergroups. For the historical background of the theory we just refer to R.I.~Jewett's fundamental paper~\cite{J} and the monograph~\cite{BH} by W.R.~Bloom and H.~Heyer. In fact, hypergroups arose in the theory of second order dif\/ferential equations and developed to be of signif\/icant applicability in probability theory where the hypergroup convolution of measure ref\/lects a stochastic operation in the basic space of the hypergroup. Nowadays hypergroup structures are studied within various frameworks from non-commutative duality of groups to quantum groups and bimodules.

Since every investigation of the structures of hypergroups is oriented on the search of new, probably large examples, aspects of a partial solution to the structure problem are extension of hypergroups~\cite{HKKK,HK1}, a cohomology theory for hypergroups~\cite{HK2} and imprimitivity of representations of hypergroups~\cite{HK3}. There are interesting results on hypergroup structures arising from dual objects of a hypergroup including the group case~\cite{HK4}. Recent research on the structure of hypergroups relies on the application of induced characters~\cite{HKY,H}, hyperf\/ields~\cite{HKTY2} and compact hypergroup pairs~\cite{HKTY1}. At this point we can outline our new results.

Let $H$ be a strong compact hypergroup satisfying the second axiom of countability, and let $H_0$ be a subhypergroup of $H$ with $|H/H_0| < +\infty$. By $\mathbb{Z}_q(2)$ we denote the $q$-deformation of~$\mathbb{Z}(2)$, and the hats on $H$ and $H_0$ signify their duals. In~\cite{HKY} the notion of an induced character of a~f\/inite-dimensional representation of $H$ was introduced and studied in detail. The results obtained in that paper enable us in the present work to discuss character hypergroups of the type \mbox{$\mathcal{K}\big(\hat{H} \cup \widehat{H_0}, \mathbb{Z}_q(2)\big)$} which generalize those introduced in~\cite{HKTY1}. The admissible {\it group} pair of~\cite{HKTY1} will now be replaced by an admissible {\it hypergroup} pair, and the hypergroup structure of \mbox{$\mathcal{K}\big(\hat{H} \cup \widehat{H_0}, \mathbb{Z}_q(2)\big)$} will be characterized by the hypergroup pair $(H, H_0)$ (Theorem~\ref{Theorem3.8}). Applications to semi-direct product hypergroups follow (Theorem~\ref{Theorem4.7}), and a list of new hypergroups appears in Section~\ref{section5}.

\section{Preliminaries}\label{section2}

In order to facilitate the reader's access to the problem discussed in this paper we recapitulate the notion of a hypergroup and of a few often applied facts. Details of the theory of hypergroups and standard examples can be found in the seminal paper~\cite{J} of R.I.~Jewett and in the monograph~\cite{BH} of W.R.~Bloom and H.~Heyer.

For a given locally compact (Hausdorf\/f) space $X$ we denote by $C^b(X)$ the space of bounded continuous functions on $X$, and by $C_c(X)$ and $C_0(X)$ its subspaces of functions with compact support or of functions vanishing at inf\/inity respectively. For each compact subset~$K$ of~$X$ let~$C_K(X)$ be the subset of functions $f\in C_c(X)$ with supp($f$)$\subset K$. By $M(X)$ we denote the set of Radon measures on $X$ def\/ined as linear functionals on~$C_c(X)$ whose restriction to each~$C_K(X)$ is continuous with respect to the topology of uniform convergence. $M^b(X)$ symbolizes the set of bounded measures on~$X$. In fact, $M^b(X)$ is the dual of the Banach space~$C_0(X)$, and it is furnished with the norm
\begin{gather*}\mu \mapsto \|\mu\| := \sup \{|\mu(f)|\colon \| f\|\leq 1\}.\end{gather*}
Moreover, we shall refer to the subspaces $M_c(X)$ and $M^1(X)$ of measures with compact support or probability measures on $X$ respectively.

Finally, $M^1_c(X) := M^1(X)\cap M_c(X)$. We denote the Dirac measure in $x\in X$ by $\varepsilon_x$.

\looseness=1 A {\it hypergroup} is a locally compact (Hausdorf\/f) space $H$ together with a weakly con\-ti\-nuous associative and bilinear convolution $*$ in the Banach space $M^b(H)$ satisfying the following \mbox{axioms}:
\begin{description}\itemsep=0pt
\item[{\rm (HG1)}]
For all $x,y \in H$, $\varepsilon_x * \varepsilon_y$ belongs to $M^1_c(H)$.

\item[{\rm (HG2)}]
There exist a neutral element $e \in H$ such that
\begin{gather*}
\varepsilon_x * \varepsilon_e = \varepsilon_e * \varepsilon_x = \varepsilon_x
\end{gather*}
for all $x \in H$, and a continuous involution
\begin{gather*}
x \longmapsto x^-
\end{gather*}
in $H$ satisfying
\begin{gather*}
e \in \operatorname{supp}(\varepsilon_x * \varepsilon_y)~~\text{if and only if}~~y = x^-
\end{gather*}
as well as
\begin{gather*}
(\varepsilon_x * \varepsilon_y)^- = \varepsilon_{y^-} * \varepsilon_{x^-}
\end{gather*}
whenever $x,y \in H$.

\item[{\rm (HG3)}]
The mapping
\begin{gather*}
(x,y) \longmapsto \operatorname{supp}(\varepsilon_x * \varepsilon_y)
\end{gather*}
from $H \times H$ into the space of compact subsets of $H$ equipped with the Michael topology is continuous.
\end{description}

As a consequence of the weak continuity and bilinearity the convolution of arbitrary bounded measures on $H$ is uniquely determined by the convolution of Dirac measures. In other words
\begin{gather*}\mu*\nu = \int_H\int_H \varepsilon_x*\varepsilon_y \mu(dx)\nu(dy),\end{gather*}
where $\mu, \nu\in M^b(H)$.

A hypergroup $H$ is called {\it commutative} if its convolution is commutative. Clearly, locally compact groups are hypergroups. Also double coset spaces $G//L$ arising from Gelfand pairs $(G,L)$ are (commutative) hypergroups. Given a hypergroup $H$ one can introduce subhypergroups, quotient hypergroups, direct and semi-direct product hypergroups (for the latter notion see \cite{HK3,W}), and hypergroup joins.

Every compact hypergroup $H$ has the normalized {\it Haar measure} $\omega_H\in M(H)$ which is inva\-riant with respect to the translation
\begin{gather*}f\mapsto f_x, \qquad x\in H,\end{gather*}
where
\begin{gather*}
f_x(y) := \varepsilon_x * \varepsilon_y(f) = \int_H f(z)(\varepsilon_x * \varepsilon_y)(dz)
\end{gather*}
for all $y\in H$.

Let $(H, *)$ and $(L, \circ)$ be two hypergroups with convolutions $*$ and $\circ$ as well with neutral elements $\varepsilon_H$ and $\varepsilon_L$ respectively. A continuous mapping $\varphi\colon H \rightarrow L$ is called a {\it hypergroup homomorphism} if $\varphi(\varepsilon_H)= \varepsilon_L$ and if $\varphi$ is the unique linear weakly continuous extension from~$M^b(H)$ to~$M^b(L)$ satisfying the following conditions:
\begin{enumerate}\itemsep=0pt
\item[(1)]
$\varphi(\varepsilon_x) = \varepsilon_{\varphi(x)}$,
\item[(2)]
$\varphi(\varepsilon_x^-) = \varphi(\varepsilon_x)^- $,
\item[(3)]
$\varphi (\varepsilon_x * \varepsilon_y) = \varepsilon_{\varphi(x)} \circ \varepsilon_{\varphi(y)}$,
\end{enumerate}
\noindent
whenever $x,y \in H$.

If, in addition, $\varphi$ is a homeomorphism from $H$ onto $L$, it is called an {\it isomorphism} from $H$ onto $L$, and in the case $L=H$ it is called an {\it automorphism} of $H$. The set Aut$(H)$ of all automorphisms of $H$ becomes a topological group furnished with the weak topology of $M^b(H)$.

An {\it action} of a locally compact group $G$ on a hypergroup $H$ is a continuous homomorphism from $G$ into Aut$(H)$.

Given an action $\alpha$ of $G$ on $H$ there is the notion of a semi-direct product hypergroup $K=H\rtimes_{\alpha}G$ which in general is a non-commutative hypergroup, ef\/f\/iciently applied all over in our work.

Let $H$ be a hypergroup, and let $\mathcal{H}$ be a (separable) Hilbert space with inner product $\langle\cdot, \cdot\rangle$. By $\mathcal{B}(\mathcal{H})$ we denote the Banach $*$-algebra of bounded linear operator on $\mathcal{H}$. A $*$-homomorphism $\pi \colon M^b(H)\longrightarrow \mathcal{B}(\mathcal{H})$ is called a {\it representation} of $H$ if
\begin{gather*}\pi(\varepsilon_e)=1\end{gather*}
and for all $u, v \in \mathcal{H}$ the mapping
\begin{gather*}\mu \mapsto \langle\pi(\mu)u, v\rangle\end{gather*}
is continuous on $M^b(H)$.

In the sequel we shall deal with classes of representations and of irreducible representations of $H$ under unitary equivalence.

Now let $H$ be a compact hypergroup with a countable basis of its topology. $\hat{H}$ will denote the set of all equivalence classes of irreducible representations of $H$. $H$ is said to be of {\it strong type} if $\hat{H}$ carries a hypergroup structure. If $H$ is commutative, more structure is available. In this case $\hat{H}$ consists of characters of $H$ which are def\/ined as nonvanishing functions $\chi\in C^b(H)$ satisfying the equality
\begin{gather*}(\varepsilon_x * \varepsilon_y^-)(\chi) = \chi(x)\overline{\chi}(y)\end{gather*}
valid for all $x,y \in H$. Once $\hat{H}$ is a hypergroup, the double dual $\hat{\hat{H}}$ can be formed, and the identif\/ication $\hat{\hat{H}}\cong H$ def\/ines {\it Pontryagin hypergroups}.

Returning to an arbitrary compact hypergroup $H$ and a closed subhypergroup $H_0$ of $H$, for a~representation $\pi_0$ of $H_0$ with representing Hilbert space $\mathcal{H}(\pi_0)$ one introduces the {\it representation}
\begin{gather*}\pi := \operatorname{ind}_{H_0}^H \pi_0\end{gather*}
{\it induced} by $\pi_0$ from $H_0$ to $H$ as follows:
\begin{gather*}
\mathcal{H}(\pi) := \big\{\xi \in L^2(H,\mathcal{H}(\pi_0)) \colon
(\varepsilon_{h_0} * \varepsilon_x)(\xi) = \pi_0(h_0)\xi(x) \ {\rm for \ all} \ h_0\in H_0 \big\}\end{gather*}
and
\begin{gather*}(\pi(h)\xi)(x) := (\varepsilon_x * \varepsilon_h)(\xi)\end{gather*}
for all $\xi \in \mathcal{H}(\pi)$, $x,h\in H$.

For further details on induced representations, see~\cite{HK3,HKY}.

\section{Hypergroups related to admissible pairs}\label{section3}

Let $H$ be a strong compact hypergroup which satisf\/ies the second axiom of countability, and $\hat{H}$ its dual. Then
\begin{gather*}
\mathcal{K}(\hat{H})=\big\{{\operatorname{ch}}(\pi) \colon \pi\in \hat{H}\big\}
\end{gather*}
is a countable discrete commutative hypergroup, where{\samepage
\begin{gather*}
\operatorname{ch}(\pi)(h)=\frac{1}{\dim\pi}\operatorname{tr}(\pi(h)).
\end{gather*}
for all $\pi \in \hat{H}$, $h \in H$.}

Now, let $H_0$ be a subhypergroup of $H$ which is assumed to be also of strong type and such that $|H/H_0| < +\infty$. For $\tau \in \widehat{H_0}$ the induced representation $\operatorname{ind}_{H_0}^H\tau$ of $\tau$ from $H_0$ to $H$ is f\/inite-dimensional and decomposes as
\begin{gather*}
\operatorname{ind}_{H_0}^H\tau\cong \pi_1\oplus \cdots\oplus \pi_m,
\end{gather*}
where $\pi_1, \dots, \pi_m \in \hat{H}~(m \geq 1)$. The {\it induced character} of $\operatorname{ch}(\tau)$ is def\/ined as
\begin{gather*}
\operatorname{ind}_{H_0}^H \operatorname{ch}(\tau) := \frac{d(\pi_1)}{d(\pi)} \operatorname{ch}(\pi_1) + \cdots + \frac{d(\pi_m)}{d(\pi)}\operatorname{ch}(\pi_m),
\end{gather*}
where $d(\pi_j)$ for $j=1, \dots, m$ is the hyperdimension of $\pi_j$ in the sense of Vrem \cite{V} and
\begin{gather*}
d(\pi) := d(\pi_1) + \cdots + d(\pi_m).
\end{gather*}

For $\tau_i, \tau_j \in \widehat{H_0}$, $\operatorname{ch}(\tau_i)\operatorname{ch}(\tau_j) \in M_c^1\big(\mathcal{K}(\widehat{H_0})\big)$ such that
\begin{gather*}
\operatorname{ch}(\tau_i)\operatorname{ch}(\tau_j) = a_1 \operatorname{ch}(\tau_1) + \cdots + a_{\ell} \operatorname{ch}(\tau_{\ell}),
\end{gather*}
with $a_k > 0$ $(k=1, \dots, \ell)$ and $a_1 + \cdots + a_{\ell} = 1$. Concerning characters induced from $H_0$ to $H$ and the following def\/inition see~\cite{HKY}.

\begin{Definition}\label{Definition3.1}
For $\tau_i, \tau_j \in \widehat{H_0}$,
\begin{gather*}
\operatorname{ind}_{H_0}^H (\operatorname{ch}(\tau_i)\operatorname{ch}(\tau_j)):= a_1\operatorname{ind}_{H_0}^H \operatorname{ch}(\tau_1) + \cdots +a_{\ell} \operatorname{ind}_{H_0}^H \operatorname{ch}(\tau_{\ell}).
\end{gather*}
\end{Definition}

Our main objective of study will be formulated in the subsequent

\begin{Definition}\label{Definition3.2}
Let $\mathbb{Z}_q(2)$ be a hypergroup of order 2 with parameter $q \in (0,1]$. The {\it twisted convolution} $*=*_q $
on the space
\begin{gather*}
\mathcal{K}\big(\hat{H} \cup \widehat{H_0},\mathbb{Z}_q(2)\big)
:= \big\{(\operatorname{ch}(\pi), \circ), (\operatorname{ch}(\tau), \bullet) \colon \pi \in \hat{H}, \tau \in \widehat{H_0}\big\},
\end{gather*}
associated with $\mathbb{Z}_q(2)$ is given as follows:
\begin{alignat*}{3}
&(1) \quad && (\operatorname{ch}(\pi_i),\circ)*(\operatorname{ch}(\pi_j),\circ):=(\operatorname{ch}(\pi_i)\operatorname{ch}(\pi_j),\circ),&\\
& (2)\quad && (\operatorname{ch}(\pi),\circ)*(\operatorname{ch}(\tau),\bullet):= \big(\big({\operatorname{res}}_{H_0}^H \operatorname{ch}(\pi)\big)\operatorname{ch}(\tau),\bullet\big),&\\
&(3)\quad && (\operatorname{ch}(\tau),\bullet)*(\operatorname{ch}(\pi),\circ):= \big({\operatorname{ch}}(\tau)\big({\operatorname{res}}_{H_0}^H \operatorname{ch}(\pi)\big),\bullet\big),&\\
&(4)\quad && (\operatorname{ch}(\tau_i),\bullet)*(\operatorname{ch}(\tau_j),\bullet):=
q\big({\operatorname{ind}}_{H_0}^H(\operatorname{ch}(\tau_i)\operatorname{ch}(\tau_j)),\circ\big) +(1-q)(\operatorname{ch}(\tau_i)\operatorname{ch}(\tau_j),\bullet).&
\end{alignat*}
\end{Definition}

For details on deformation of hypergroups see \cite{KTY}.

\begin{Definition}\label{Definition3.3}
Let $(H, H_0)$ be a pair of consisting of a compact hypergroup $H$ and a closed subhypergroup $H_0$ of $H$. We call $(H, H_0)$ an {\it admissible hypergroup pair} if the following conditions are satisf\/ied:
\begin{enumerate}\itemsep=0pt
\item[(1)] for $\pi \in \hat{H}$ and $\tau \in \widehat{H_0}$
\begin{gather*}
\operatorname{ind}_{H_0}^H\big(\big({\operatorname{res}}_{H_0}^H \operatorname{ch}(\pi)\big)\operatorname{ch}(\tau)\big) = \operatorname{ch}(\pi) \operatorname{ind}_{H_0}^H \operatorname{ch}(\tau),
\end{gather*}
\item[(2)] for $\tau \in \widehat{H_0}$
\begin{gather*}
\operatorname{res}_{H_0}^H\big({\operatorname{ind}}_{H_0}^H \operatorname{ch}(\tau)\big) = \operatorname{ch}(\tau) \operatorname{res}_{H_0}^H\big({\operatorname{ind}}_{H_0}^H \operatorname{ch}(\tau_0)\big),
\end{gather*}
where $\tau_0$ is the trivial representation of~$H_0$.
\end{enumerate}
\end{Definition}

\begin{Remark}\label{Remark3.4}
If a pair $(G,G_0)$ consisting of a~compact group and a closed subgroup $G_0$ is admissible in the sense of~\cite{HKTY1}, then it is an admissible hypergroup pair.
\end{Remark}

\begin{Lemma}\label{Lemma3.5}
If $(H, H_0)$ is an admissible hypergroup pair, the following formulae hold:
\begin{enumerate}\itemsep=0pt
\item[$(1)$] for $\pi \in \hat{H}$ and $\tau_i, \tau_j \in \widehat{H_0}$
\begin{gather*}
\operatorname{ind}_{H_0}^H\big(\big({\operatorname{res}}_{H_0}^H \operatorname{ch}(\pi)\big)\operatorname{ch}(\tau_i)\operatorname{ch}(\tau_j)\big) = \operatorname{ch}(\pi) \operatorname{ind}_{H_0}^H (\operatorname{ch}(\tau_i)\operatorname{ch}(\tau_j)),
\end{gather*}
\item[$(2)$] for $\tau_i, \tau_j \in \widehat{H_0}$
\begin{gather*}
\operatorname{res}_{H_0}^H\big({\operatorname{ind}}_{H_0}^H \operatorname{ch}(\tau_i)\operatorname{ch}(\tau_j)\big) = \operatorname{ch}(\tau_i)\operatorname{ch}(\tau_j) \operatorname{res}_{H_0}^H\big({\operatorname{ind}}_{H_0}^H \operatorname{ch}(\tau_0)\big).
\end{gather*}
\end{enumerate}
\end{Lemma}

\begin{proof}It is easy to see the desired formulae by the def\/inition of $\operatorname{ind}_{H_0}^H (\operatorname{ch}(\tau_i)\operatorname{ch}(\tau_j))$.
\end{proof}

\begin{Proposition}\label{Proposition3.6}
If a strong compact hypergroup $H$ together with a strong subhypergroup $H_0$ of $H$ with $|H/H_0|< + \infty$ forms an admissible hypergroup pair, then the following associativity relations hold. For $\pi_i, \pi_j, \pi_k, \pi \in \hat{H}$ and $\tau_i, \tau_j, \tau_k, \tau \in \widehat{H_0}$:
\begin{alignat*}{3}
&(A1)\quad && ((\operatorname{ch}(\pi_i), \circ) * (\operatorname{ch}(\pi_j), \circ)) * (\operatorname{ch}(\pi_k), \circ)
= (\operatorname{ch}(\pi_i), \circ) * ((\operatorname{ch}(\pi_j), \circ) * (\operatorname{ch}(\pi_k), \circ)),& \\
&(A2) \quad && ((\operatorname{ch}(\tau), \bullet) * (\operatorname{ch}(\pi_i), \circ)) * (\operatorname{ch}(\pi_j), \circ)
= (\operatorname{ch}(\tau), \bullet) * ((\operatorname{ch}(\pi_i), \circ) * (\operatorname{ch}(\pi_j), \circ)),&\\
&(A3)\quad && ((\operatorname{ch}(\tau_i), \bullet) * (\operatorname{ch}(\tau_j), \bullet)) * (\operatorname{ch}(\pi), \circ)
=(\operatorname{ch}(\tau_i), \bullet) * ((\operatorname{ch}(\tau_j), \bullet) * (\operatorname{ch}(\pi), \circ)), &\\
&(A4)\quad && ((\operatorname{ch}(\tau_i), \bullet) * (\operatorname{ch}(\tau_j), \bullet)) * (\operatorname{ch}(\tau_k), \bullet)
=(\operatorname{ch}(\tau_i), \bullet) * ((\operatorname{ch}(\tau_j), \bullet) * (\operatorname{ch}(\tau_k), \bullet)).&
\end{alignat*}
\end{Proposition}

\begin{proof} $(A1)$ is clear by the assumption that $H$ is strong, i.e., $\mathcal{K}(\hat{H})$ is a hypergroup.

$(A2)$ For $\tau \in \widehat{H_0}$ and $\pi_i, \pi_j \in \hat{H}$,
\begin{gather*}
((\operatorname{ch}(\tau), \bullet) * (\operatorname{ch}(\pi_i), \circ)) * (\operatorname{ch}(\pi_j), \circ) = \big({\operatorname{ch}}(\tau) \operatorname{res}_{H_0}^H \operatorname{ch}(\pi_i), \bullet\big) * (\operatorname{ch}(\pi_j), \circ)\\
\hphantom{((\operatorname{ch}(\tau), \bullet) * (\operatorname{ch}(\pi_i), \circ)) * (\operatorname{ch}(\pi_j), \circ)}{} = \big({\operatorname{ch}}(\tau) \big({\operatorname{res}}_{H_0}^H \operatorname{ch}(\pi_i)\big) \big({\operatorname{res}}_{H_0}^H \operatorname{ch}(\pi_j)\big) , \bullet\big).
\end{gather*}
On the other hand,
\begin{gather*}
(\operatorname{ch}(\tau), \bullet) * ((\operatorname{ch}(\pi_i), \circ) * (\operatorname{ch}(\pi_j), \circ))
= (\operatorname{ch}(\tau), \bullet) * (\operatorname{ch}(\pi_i) \operatorname{ch}(\pi_j), \circ)\\
\hphantom{(\operatorname{ch}(\tau), \bullet) * ((\operatorname{ch}(\pi_i), \circ) * (\operatorname{ch}(\pi_j), \circ))}{}
= (\operatorname{ch}(\tau) \operatorname{res}_{H_0}^H (\operatorname{ch}(\pi_i) \operatorname{ch}(\pi_j)), \bullet) \\
\hphantom{(\operatorname{ch}(\tau), \bullet) * ((\operatorname{ch}(\pi_i), \circ) * (\operatorname{ch}(\pi_j), \circ))}{}
= \big({\operatorname{ch}}(\tau) \big({\operatorname{res}}_{H_0}^H \operatorname{ch}(\pi_i)\big) \big({\operatorname{res}}_{H_0}^H \operatorname{ch}(\pi_j)\big), \bullet\big).
\end{gather*}

$(A3)$ For $\tau_i, \tau_j \in \widehat{H_0}$ and $\pi \in \hat{H}$,
\begin{gather*}
((\operatorname{ch}(\tau_i), \bullet) * (\operatorname{ch}(\tau_j), \bullet)) * (\operatorname{ch}(\pi), \circ)\\
\qquad{} = \big(q\big({\operatorname{ind}}_{H_0}^H(\operatorname{ch}(\tau_i)\operatorname{ch}(\tau_j)), \circ\big)+(1-q)(\operatorname{ch}(\tau_i)\operatorname{ch}(\tau_j),\bullet)\big)* (\operatorname{ch}(\pi), \circ)\\
\qquad{}= q\big({\operatorname{ind}}_{H_0}^H(\operatorname{ch}(\tau_i)\operatorname{ch}(\tau_j)), \circ\big)* (\operatorname{ch}(\pi), \circ)+(1-q)(\operatorname{ch}(\tau_i)\operatorname{ch}(\tau_j),\bullet)* (\operatorname{ch}(\pi), \circ)\\
\qquad{} = q\big({\operatorname{ind}}_{H_0}^H(\operatorname{ch}(\tau_i)\operatorname{ch}(\tau_j))\operatorname{ch}(\pi), \circ\big)+(1-q)\big({\operatorname{ch}}(\tau_i)\operatorname{ch}(\tau_j)\operatorname{res}_{H_0}^H\operatorname{ch}(\pi), \bullet\big)\\
\qquad{} =q\big({\operatorname{ind}}_{H_0}^H (\operatorname{ch}(\tau_i)\operatorname{ch}(\tau_j)\operatorname{res}_{H_0}^H \operatorname{ch}(\pi)),\circ\big)+(1-q)\big({\operatorname{ch}}(\tau_i)\operatorname{ch}(\tau_j)\operatorname{res}_{H_0}^H\operatorname{ch}(\pi), \bullet\big)\\
\qquad{}=(\operatorname{ch}(\tau_i), \bullet) * \big({\operatorname{ch}}(\tau_j)\operatorname{res}_{H_0}^H \operatorname{ch}(\pi), \bullet\big)\\
\qquad{} =(\operatorname{ch}(\tau_i), \bullet) * ((\operatorname{ch}(\tau_j), \bullet) * (\operatorname{ch}(\pi), \circ)).
\end{gather*}

$(A4)$ For $\tau_i, \tau_j, \tau_k \in \widehat{H_0}$
\begin{gather*}
((\operatorname{ch}(\tau_i), \bullet) * (\operatorname{ch}(\tau_j), \bullet)) * (\operatorname{ch}(\tau_k), \bullet)\\
\qquad{} = \big(q \big({\operatorname{ind}}_{H_0}^H(\operatorname{ch}(\tau_i)\operatorname{ch}(\tau_j)),\circ\big)+(1-q)(\operatorname{ch}(\tau_i)\operatorname{ch}(\tau_j),\bullet)\big) * (\operatorname{ch}(\tau_k), \bullet)\\
\qquad{}= q\big({\operatorname{ind}}_{H_0}^H(\operatorname{ch}(\tau_i)\operatorname{ch}(\tau_j)),\circ\big)*(\operatorname{ch}(\tau_k), \bullet)+(1-q)(\operatorname{ch}(\tau_i)\operatorname{ch}(\tau_j),\bullet) * (\operatorname{ch}(\tau_k), \bullet)\\
\qquad {}= q\big({\operatorname{res}}_{H_0}^H\big({\operatorname{ind}}_{H_0}^H(\operatorname{ch}(\tau_i)\operatorname{ch}(\tau_j))\big)\operatorname{ch}(\tau_k), \bullet\big)\\
\qquad\quad{} + (1-q)q\big({\operatorname{ind}}_{H_0}^H (\operatorname{ch}(\tau_i)\operatorname{ch}(\tau_j)\operatorname{ch}(\tau_k)), \circ\big)+ (1-q)^2(\operatorname{ch}(\tau_i)\operatorname{ch}(\tau_j)\operatorname{ch}(\tau_k), \bullet)\\
\qquad{}= q\big({\operatorname{res}}_{H_0}^H\big({\operatorname{ind}}_{H_0}^H \operatorname{ch}(\tau_0)\big)(\operatorname{ch}(\tau_i)\operatorname{ch}(\tau_j)\operatorname{ch}(\tau_k)), \bullet\big)\\
\qquad\quad{} + (1-q)q\big({\operatorname{ind}}_{H_0}^H (\operatorname{ch}(\tau_i)\operatorname{ch}(\tau_j)\operatorname{ch}(\tau_k)), \circ\big)+ (1-q)^2(\operatorname{ch}(\tau_i)\operatorname{ch}(\tau_j)\operatorname{ch}(\tau_k), \bullet).
\end{gather*}
This implies the associativity:
\begin{gather*}
((\operatorname{ch}(\tau_i), \bullet) * (\operatorname{ch}(\tau_j), \bullet)) * (\operatorname{ch}(\tau_k), \bullet) =
(\operatorname{ch}(\tau_i), \bullet) *( (\operatorname{ch}(\tau_j), \bullet) * (\operatorname{ch}(\tau_k), \bullet)).\tag*{\qed}
\end{gather*}
\renewcommand{\qed}{}
\end{proof}

\begin{Proposition}\label{Proposition3.7}
Let $H$ be a strong compact hypergroup and $H_0$ a subhypergroup of $H$ which is also of strong type and such that $|H/H_0| < +\infty$. If $\mathcal{K}\big(\hat{H} \cup \widehat{H_0}, \mathbb{Z}_q(2)\big)$ is a hypergroup, then $(H, H_0)$ is an admissible hypergroup pair.
\end{Proposition}

\begin{proof} (1) By the associativity $(A3)$ we have
\begin{gather*}
(\operatorname{ch}(\tau), \bullet) *( (\operatorname{ch}(\tau_0), \bullet) * (\operatorname{ch}(\pi), \circ))
=((\operatorname{ch}(\tau), \bullet) * (\operatorname{ch}(\tau_0), \bullet) )* (\operatorname{ch}(\pi), \circ),
\end{gather*}
and the formulae
\begin{gather*}(\operatorname{ch}(\tau), \bullet) *( (\operatorname{ch}(\tau_0), \bullet) * (\operatorname{ch}(\pi), \circ))\\
\qquad{} = q\big({\operatorname{ind}}_{H_0}^H \big({\operatorname{ch}}(\tau)\operatorname{res}_{H_0}^H \operatorname{ch}(\pi)\big),\circ\big)+(1-q)\big({\operatorname{ch}}(\tau)\operatorname{res}_{H_0}^H\operatorname{ch}(\pi), \bullet\big)\end{gather*}
and
\begin{gather*}((\operatorname{ch}(\tau), \bullet) * (\operatorname{ch}(\tau_0), \bullet))* (\operatorname{ch}(\pi), \circ)\\
\qquad{} = q\big(\big({\operatorname{ind}}_{H_0}^H\operatorname{ch}(\tau)\big)\operatorname{ch}(\pi), \circ)+(1-q)\big({\operatorname{ch}}(\tau)\operatorname{res}_{H_0}^H\operatorname{ch}(\pi), \bullet\big).
\end{gather*}
Hence we obtain the admissibility condition (1)
\begin{gather*}
\operatorname{ind}_{H_0}^H\big(\big({\operatorname{res}}_{H_0}^H \operatorname{ch}(\pi)\big)\operatorname{ch}(\tau)\big) = \operatorname{ch}(\pi) \operatorname{ind}_{H_0}^H \operatorname{ch}(\tau).
\end{gather*}

(2) By the associativity $(A4)$
\begin{gather*}(\operatorname{ch}(\tau), \bullet) *( (\operatorname{ch}(\tau_0), \bullet) * (\operatorname{ch}(\tau_0), \bullet))\\
\qquad{} = ((\operatorname{ch}(\tau), \bullet) * (\operatorname{ch}(\tau_0), \bullet)) * (\operatorname{ch}(\tau_0), \bullet),
 ((\operatorname{ch}(\tau), \bullet) * (\operatorname{ch}(\tau_0), \bullet)) * (\operatorname{ch}(\tau_0), \bullet)\\
\qquad{} = q\big({\operatorname{res}}_{H_0}^H\big({\operatorname{ind}}_{H_0}^H\operatorname{ch}(\tau), \bullet\big)
 + (1-q)q\big({\operatorname{ind}}_{H_0}^H \operatorname{ch}(\tau)\big), \circ\big)+ (1-q)^2(\operatorname{ch}(\tau), \bullet),\end{gather*}
and
\begin{gather*}(\operatorname{ch}(\tau), \bullet) *( (\operatorname{ch}(\tau_0), \bullet) * (\operatorname{ch}(\tau_0), \bullet))\\
\qquad{} = q\big({\operatorname{ch}}(\tau)\operatorname{res}_{H_0}^H\big({\operatorname{ind}}_{H_0}^H\operatorname{ch}(\tau_0)\big), \bullet\big)
 + (1-q)q\big({\operatorname{ind}}_{H_0}^H \operatorname{ch}(\tau), \circ\big)+ (1-q)^2(\operatorname{ch}(\tau), \bullet).\end{gather*}
Hence we obtain the admissibility condition (2)
\begin{gather*}
\operatorname{res}_{H_0}^H\big({\operatorname{ind}}_{H_0}^H \operatorname{ch}(\tau)\big) = \operatorname{ch}(\tau) \operatorname{res}_{H_0}^H\big({\operatorname{ind}}_{H_0}^H \operatorname{ch}(\tau_0)\big).\tag*{\qed}
\end{gather*}\renewcommand{\qed}{}
\end{proof}

\begin{Theorem}\label{Theorem3.8} Let $H$ be a strong compact hypergroup and $H_0$ a subhypergroup of $H$ which is also of strong type and such that $|H/H_0| < +\infty$. Then $\mathcal{K}\big(\hat{H} \cup \widehat{H_0}, \mathbb{Z}_q(2)\big)$ is a hypergroup if and only if $(H, H_0)$ is an admissible hypergroup pair.
\end{Theorem}

\begin{proof} If $(H, H_0)$ is an admissible hypergroup pair, then the associativity relations $(A1)$, $(A2)$, $(A3)$ and $(A4)$ are a~consequence of Proposition~\ref{Proposition3.6}. It is easy to check the remaining axioms of a hypergroup for $\mathcal{K}\big(\hat{H} \cup \widehat{H_0}, \mathbb{Z}_q(2)\big)$. The converse statement follows from Proposition~\ref{Proposition3.7}.
\end{proof}

\begin{Remark}\label{Remark3.9} \quad
\begin{enumerate}\itemsep=0pt
\item[(1)] The above $\mathcal{K}\big(\hat{H} \cup \widehat{H_0}, \mathbb{Z}_q(2)\big)$ is a discrete commutative (at most countable) hypergroup such that the sequence
\begin{gather*}
1 \longrightarrow \mathcal{K}(\hat{H}) \longrightarrow \mathcal{K}\big(\hat{H} \cup \widehat{H_0}, \mathbb{Z}_q(2)\big)
\longrightarrow \mathbb{Z}_q(2) \longrightarrow 1
\end{gather*}
is exact.
\item[(2)] If $H_0 = H$, then $\mathcal{K}\big(\hat{H} \cup \widehat{H_0}, \mathbb{Z}_q(2)\big)$ is the direct product hypergroup $\mathcal{K}(\hat{H}) \times \mathbb{Z}_q(2)$.
\item[(3)] If $H$ is a f\/inite hypergroup and $H_0 = \{ h_0 \}$ where $h_0$ is unit of $H$, then $\mathcal{K}\big(\hat{H} \cup \widehat{H_0}, \mathbb{Z}_q(2)\big)$ is the hypergroup join $\mathcal{K}(\hat{H}) \vee \mathbb{Z}_q(2)$.
\item[(4)] If $H$ is a compact commutative hypergroup of strong type and $H_0$ is a closed subhypergroup of $H$ with $|H/H_0| < + \infty$. Then $(H, H_0)$ is always an admissible hypergroup pair and $\mathcal{K}\big(\hat{H} \cup \widehat{H_0}, \mathbb{Z}_q(2)\big)$ is a hypergroup. For more details see~\cite{HKTY2}.
\end{enumerate}
\end{Remark}

\section{Semi-direct product hypergroups}\label{section4}

We consider a non-commutative compact hypergroup, namely the semi-direct product hypergroup $K:=H\rtimes_{\alpha}G$, where $\alpha$ is an action of a~compact group $G$ on a compact commutative hypergroup $H$ of strong type. For a representation $\pi$ of $K$ we denote the restrictions of $\pi$ to $H$ and $G$ by $\rho$ and $\tau$ respectively. We shall write $\pi=\rho\odot\tau$ expressing
\begin{gather*}\pi(h,g)=\rho(h)\tau(g)\end{gather*}
for all $h\in H$, $g\in G$. The action $\hat{\alpha}$ of $G$ on $\hat{H}$ induced by $\alpha$ is given by
\begin{gather*}\hat{\alpha}_g(\chi)(h):=\chi\big(\alpha_g^{-1}(h)\big),\end{gather*}
whenever $\chi\in \hat{H}$, $g\in G$, $h\in H$. Let
\begin{gather*}G(\chi):=\{g\in G \colon \hat{\alpha}_g(\chi)=\chi\}\end{gather*}
be the stabilizer of $\chi\in \hat{H}$ under the action $\hat{\alpha}$ of $G$ on $\hat{H}$.

\begin{Proposition}\label{Proposition4.1}
Any irreducible representation $\pi$ of $K=H\rtimes_{\alpha}G$ is given by
\begin{gather*}\pi = \pi^{(\chi,\tau)} := \operatorname{ind}_{H\rtimes_{\alpha}G(\chi)}^{H\rtimes_{\alpha}G}(\chi\odot\tau),\end{gather*}
where $\chi\in \hat{H}$, $\tau \in \widehat{G(\chi)}$.
\end{Proposition}

\begin{proof} This statement is obtained by an application of the Mackey machine as stated in Theo\-rem~7.1 of~\cite{HK3}.
\end{proof}

We denote the orbit of $\chi \in \hat{H}$ by $O(\chi)$ under the action $\hat{\alpha}$ of $G$ on $\hat{H}$, i.e.,
\begin{gather*}
O(\chi) := \{ \hat{\alpha}_g(\chi) \colon g \in G \}.
\end{gather*}
Then the representation $\rho^{\chi}$ associated with $O(\chi)$ is def\/ined by
\begin{gather*}\rho^{\chi}:= \int_{O(\chi)}^{\oplus} \sigma \mu(d\sigma),\end{gather*}
where $\mu$ is the $\hat{\alpha}$-invariant probability measure supporting the orbit $O(\chi)$ of $\chi$ in $\hat{H}$. We denote the space $\{ \rho^{\chi} \colon \chi \in \hat{H} \}$ by $\hat{H}^{\hat{\alpha}}$. By the assumption that $H$ is strong we see that $\hat{H}^{\hat{\alpha}}$ is a~hypergroup which is called the {\it orbital hypergroup} of $\hat{H}$ under the action $\hat{\alpha}$. Here we note that the irreducible representation $\pi^{(\chi,\tau)}$ of $K=H\rtimes_{\alpha}G$ is written as $\pi^{(\chi,\tau)}=\rho^{\chi}\odot u^{\tau}$, where the representation $u^{\tau}$ of $G$ is given by
\begin{gather*}u^{\tau}=\operatorname{ind}_{G(\chi)}^G \tau. \end{gather*}
Moreover we note that the hyperdimension $d(\pi)$ of $\pi = \pi^{(\chi,\tau)}$ is \begin{gather*}d(\pi):=w(\chi)\dim\pi = w(\rho^\chi) \dim \tau,\end{gather*}
where $w(\chi)$ and $w(\rho^{\chi})=w(\chi)|O(\chi)|$ denote the weights of $\chi\in \hat{H}$ and $\rho^{\chi} \in \hat{H}^{\hat{\alpha}}$ respectively.

\begin{Proposition}\label{Proposition4.2}
For a representation $\pi_0 = \chi \odot \tau$ of $K_0=H\rtimes_{\alpha}G(\chi)$, where $\chi\in \hat{H}$, $\tau \in \widehat{G(\chi)}$, the character of the induced representation $\pi^{(\chi,\tau)} = \operatorname{ind}_{K_0}^K\pi_0$ of $\pi_0$ takes the form
\begin{gather*}
\operatorname{ch}\big({\operatorname{ind}}_{K_0}^K\pi_0\big)(h,g)= \int_G \chi(\alpha_s(h)) \operatorname{ch}(\tau)\big(sgs^{-1}\big) 1_{G(\chi)}\big(sgs^{-1}\big) \omega_G(ds).
\end{gather*}
Moreover,
\begin{gather*}
\big({\operatorname{ind}}_{K_0}^K\operatorname{ch}(\pi_0)\big)(h,g)= \int_G \chi(\alpha_s(h))\operatorname{ch}(\tau)\big(sgs^{-1}\big) 1_{G(\chi)}\big(sgs^{-1}\big) \omega_G(ds),
\end{gather*}
hence
\begin{gather*}
\operatorname{ind}_{K_0}^K\operatorname{ch}(\pi_0) = \operatorname{ch}\big({\operatorname{ind}}_{K_0}^K\pi_0\big)=\operatorname{ch}\big(\pi^{(\chi,\tau)}\big).
\end{gather*}
\end{Proposition}

\begin{proof} By an application of the character formulae as proved in Proposition~4.3 and Theorem~4.5 of~\cite{HKY}, we obtain the desired formulae.
\end{proof}

\begin{Proposition}[{\cite[Theorem 4.6]{HKY}}]\label{Proposition4.3}
Let $H$ be a finite commutative hypergroup of strong type. Then the induced representation $\pi=\operatorname{ind}_G^{H\rtimes_{\alpha}G}\tau$ of an irreducible representation $\tau$ of $G$ to $H\rtimes_\alpha G$ is finite-dimensional, and it is decomposed as
\begin{gather*}\pi\cong\sum_{\rho^{\chi} \in \hat{H}^{\hat{\alpha}}}\!\!{}^{\oplus}
\pi^{(\chi,\tau)}.
\end{gather*}
The character $\operatorname{ch}(\pi)$ of $\pi$ is
\begin{gather*}\operatorname{ch}(\pi)=\sum_{\rho^{\chi}\in \hat{H}^{\hat{\alpha}}} \frac{w(\rho^{\chi})}{w(\hat{H})} \operatorname{ch}\big(\pi^{(\chi,\tau)}\big).
\end{gather*}
\end{Proposition}

\begin{Definition}[\cite{HK4}]\label{Definition4.4}
The action $\alpha$ of $G$ on $H$ is said to satisfy the {\it regularity condition} (or is called regular) provided
\begin{gather*}
G(\chi_i) \cap G(\chi_j) \subset G(\chi_k)
\end{gather*}
for all $\chi_k \in \hat{H}$ such that $\chi_k \in \operatorname{supp}(\delta_{\chi_i} \hat{*} \delta_{\chi_j})$ whenever $\chi_i, \chi_j \in \hat{H}$, $k,i,j \in \{0,1,\dots, n\}$ and $\hat{*}$ symbolizes the convolution on~$\hat{H}$.
\end{Definition}

\begin{Lemma}[{\cite[Lemma 3.1]{HK4}}]\label{Lemma4.5}
If the action $\alpha$ satisfies the regularity condition, then the character set $\mathcal{K}\big(\widehat{H \rtimes_\alpha G}\big)$ of the semi-direct product hypergroup~$H \rtimes_\alpha G$ is a commutative hypergroup.
\end{Lemma}

For $g\in G$, put
\begin{gather*}\hat{H}(g) :=\big\{\sigma\in \hat{H} \colon g\in G(\sigma)\big\} = \big\{\sigma \in \hat{H} \colon \hat{\alpha}_g(\sigma) = \sigma \big\}.\end{gather*}

\begin{Proposition}\label{Proposition4.6} Let $H$ be a finite commutative hypergroup of strong type and~$G$ a compact group. Assume that the action~$\alpha$ of $G$ on $H$ satisfies the regularity condition. Then the followings hold:
\begin{enumerate}\itemsep=0pt
\item[$(1)$] $\hat{H}(g)$ is a subhypergroup of $\hat{H}$,
\item[$(2)$] for $g, t \in G$ the condition $tgt^{-1} \in G(\sigma)$ implies that $\hat{\alpha}^{-1}_t(\sigma) \in \hat{H}(g)$,
\item[$(3)$] for $\tau \in \hat{G}$
\begin{gather*}\big({\operatorname{ind}}_G^{H\rtimes_{\alpha}G}\operatorname{ch}(\tau)\big)(h,g) = \frac{w(\hat{H}(g))}{w(\hat{H})} \omega_{\hat{H}(g)}(h) \cdot \operatorname{ch}(\tau)(g),\end{gather*}
where $\omega_{\hat{H}(g)}$ is the normalized Haar measure of $\hat{H}(g)$,
\item[$(4)$] for $\tau \in \hat{G}$
\begin{gather*}\operatorname{res}_G^{H\rtimes_{\alpha}G}\big({\operatorname{ind}}_G^{H\rtimes_{\alpha}G}\operatorname{ch}(\tau)\big)(g)
= \frac{w(\hat{H}(g))}{w(\hat{H})} \cdot \operatorname{ch}(\tau)(g).\end{gather*}
\end{enumerate}
\end{Proposition}

\begin{proof} (1) It is clear that $\sigma^- \in \hat{H}(g)$ for $\sigma \in \hat{H}(g)$. We show that $\operatorname{supp}(\sigma_i \hat{*} \sigma_j) \subset \hat{H}(g)$ for $\sigma_i, \sigma_j \in \hat{H}(g)$. The condition $\sigma_i, \sigma_j \in \hat{H}(g)$ implies that $g \in G(\sigma_i) \cap G(\sigma_j)$. Take $\sigma \in  \operatorname{supp}(\sigma_i \hat{*} \sigma_j)$. Then by the regularity condition
\begin{gather*} G(\sigma_i) \cap G(\sigma_j) \subset G(\sigma)\end{gather*}
we see that $g \in G(\sigma)$ which implies that $\sigma \in \hat{H}(g)$.

(2) By the condition $tgt^{-1} \in G(\sigma)$ we see that $\hat{\alpha}_{tgt^{-1}}(\sigma) = \sigma$.
Then we obtain
\begin{gather*}\hat{\alpha}_{t}(\hat{\alpha}_{g}(\hat{\alpha}_{t^{-1}}(\sigma))) = \sigma,\end{gather*}
so that
\begin{gather*}\hat{\alpha}_{g}\big(\hat{\alpha}^{-1}_t(\sigma)\big)= {\hat{\alpha}_t}^{-1}(\sigma),\end{gather*}
which means that $\hat{\alpha_{t}}^{-1}(\sigma) \in \hat{H}(g)$.

(3) Applying Propositions~\ref{Proposition4.2} and~\ref{Proposition4.3} we obtain for $\tau \in \hat{G}$
\begin{gather*}
\big({\operatorname{ind}}^{H\rtimes_{\alpha}G}_G \operatorname{ch}(\tau)\big)(h,g) =
\sum_{\rho^{\chi}\in \hat{H}^{\hat{\alpha}}}\frac{w(\rho^{\chi})}{w(\hat{H})}\operatorname{ch}\big(\pi^{(\chi,\tau)}\big)(h,g)\\
\hphantom{\big({\operatorname{ind}}^{H\rtimes_{\alpha}G}_G \operatorname{ch}(\tau)\big)(h,g) }{} =
\sum_{\rho^{\chi}\in \hat{H}^{\hat{\alpha}}}\frac{w(\rho^{\chi})}{w(\hat{H})}
\int_G\chi(\alpha_s(h))\operatorname{ch}(\tau)\big(sgs^{-1}\big)1_{G(\chi)}\big(sgs^{-1}\big)d\omega_G(s)\\
 \hphantom{\big({\operatorname{ind}}^{H\rtimes_{\alpha}G}_G \operatorname{ch}(\tau)\big)(h,g) }{}=
\sum_{\rho^{\chi}\in \hat{H}^{\hat{\alpha}}}\frac{w(\rho^{\chi})}{w(\hat{H})}
\int_G\chi(\alpha_s(h))1_{G(\chi)}\big(sgs^{-1}\big)d\omega_G(s)\cdot \operatorname{ch}(\tau)(g) \\
\hphantom{\big({\operatorname{ind}}^{H\rtimes_{\alpha}G}_G \operatorname{ch}(\tau)\big)(h,g) }{}=
\sum_{\rho^{\chi}\in \hat{H}^{\hat{\alpha}}}\frac{|O(\chi)|w(\chi)}{w(\hat{H})}
\cdot\frac{1}{|O(\chi)|} \sum_{\sigma\in O(\chi)}\sigma(h)1_{G(\sigma)}(g)\cdot \operatorname{ch}(\tau)(g) \\
\hphantom{\big({\operatorname{ind}}^{H\rtimes_{\alpha}G}_G \operatorname{ch}(\tau)\big)(h,g) }{} =
\sum_{\sigma\in \widehat{H}}\frac{w(\sigma)}{w(\hat{H})}\sigma(h)1_{G(\sigma)}(g)\cdot \operatorname{ch}(\tau)(g) \\
 \hphantom{\big({\operatorname{ind}}^{H\rtimes_{\alpha}G}_G \operatorname{ch}(\tau)\big)(h,g) }{}=
\frac{w(\hat{H}(g))}{w(\hat{H})}\sum_{\sigma\in \hat{H}(g)}\frac{w(\sigma)}{w(\hat{H}(g))}\sigma(h)\cdot \operatorname{ch}(\tau)(g) \\
\hphantom{\big({\operatorname{ind}}^{H\rtimes_{\alpha}G}_G \operatorname{ch}(\tau)\big)(h,g) }{}=
\frac{w(\hat{H}(g))}{w(\hat{H})} \omega_{\hat{H}(g)}(h) \cdot \operatorname{ch}(\tau)(g).
\end{gather*}

(4) For $\tau \in \hat{G}$
\begin{gather*}
 \operatorname{res}_G^{H\rtimes_{\alpha}G}\big({\operatorname{ind}}_G^{H\rtimes_{\alpha}G}\operatorname{ch}(\tau)\big)(g) = \big({\operatorname{ind}}_G^{H\rtimes_{\alpha}G}\operatorname{ch}(\tau)\big)(h_0,g)\\
\hphantom{\operatorname{res}_G^{H\rtimes_{\alpha}G}\big({\operatorname{ind}}_G^{H\rtimes_{\alpha}G}\operatorname{ch}(\tau)\big)(g)}{} = \frac{w(\hat{H}(g))}{w(\hat{H})} \omega_{\hat{H}(g)}(h_0) \cdot \operatorname{ch}(\tau)(g) = \frac{w(\hat{H}(g))}{w(\hat{H})} \cdot \operatorname{ch}(\tau)(g).\tag*{\qed}
\end{gather*}\renewcommand{\qed}{}
\end{proof}

\begin{Theorem}\label{Theorem4.7} Let $H$ be a finite commutative hypergroup of strong type and $G$ a compact group. Suppose that the action~$\alpha$ of~$G$ on~$H$ satisfies the regularity condition. Then the pair $(H\rtimes_{\alpha}G, G)$ is an admissible hypergroup pair and $\mathcal{K}\big(\widehat{H\rtimes_{\alpha}G} \cup \hat{G},\mathbb{Z}_q(2)\big)$ is a discrete commutative hypergroup.
\end{Theorem}

\begin{proof} By the Mackey machine an irreducible representation $\pi$ of $K=H\rtimes_{\alpha}G$ is given by
\begin{gather*}\pi = \pi^{(\chi,\tau_1)} = \operatorname{ind}_{H\rtimes_{\alpha}G(\chi)}^{H\rtimes_{\alpha}G}(\chi\odot\tau_1),\end{gather*}
where $\chi\in \hat{H}, \tau_1 \in \widehat{G(\chi)}$ and
\begin{gather*}
\pi^{(\chi,\tau_1)} (h,g)
= \int_G \chi(\alpha_s(h))\operatorname{ch}(\tau_1)\big(sgs^{-1}\big) 1_{G(\chi)}\big(sgs^{-1}\big) \omega_G(ds),
\end{gather*}
by Propositions~\ref{Proposition4.1} and~\ref{Proposition4.2}.

(1) For $\tau \in \hat{G}$ and $\pi^{(\chi,\tau_1)} \in \hat{K} = \widehat{H\rtimes_{\alpha}G}$, applying Proposition~\ref{Proposition4.6},
\begin{gather*}
\big({\operatorname{ind}}_G^{H\rtimes_{\alpha}G}
 \operatorname{ch}(\tau) \cdot \operatorname{ch}\big( \pi^{(\chi,\tau_1)}\big)\big)(h,g)
= \big({\operatorname{ind}}_G^{H\rtimes_{\alpha}G}
 \operatorname{ch}(\tau)\big)(h,g) \operatorname{ch}\big( \pi^{(\chi,\tau_1)}\big)(h,g) \\
\qquad{} = \frac{w(\hat{H}(g))}{w(\hat{H})} \omega_{\hat{H}(g)}(h)
\cdot \operatorname{ch}(\tau)(g) \int_G \chi(\alpha_s(h))\operatorname{ch}(\tau_1)\big(sgs^{-1}\big) 1_{G(\chi)}\big(sgs^{-1}\big) \omega_G(ds)\\
\qquad{}= \frac{w(\hat{H}(g))}{w(\hat{H})}
\cdot \operatorname{ch}(\tau)(g) \int_G \omega_{\hat{H}(g)}(h) \hat{\alpha}^{-1}_s(\chi)(h)\operatorname{ch}(\tau_1)\big(sgs^{-1}\big) 1_{G(\chi)}\big(sgs^{-1}\big) \omega_G(ds)\\
\qquad{}= \frac{w(\hat{H}(g))}{w(\hat{H})}
\cdot \operatorname{ch}(\tau)(g) \int_G (\omega_{\hat{H}(g)} \hat{\alpha}^{-1}_s (\chi))(h)\operatorname{ch}(\tau_1)\big(sgs^{-1}\big) 1_{G(\chi)}\big(sgs^{-1}\big) \omega_G(ds)\\
\qquad{}= \frac{w(\hat{H}(g))}{w(\hat{H})}
\cdot \operatorname{ch}(\tau)(g) \int_G \omega_{\hat{H}(g)}(h )\operatorname{ch}(\tau_1)\big(sgs^{-1}\big) 1_{G(\chi)}\big(sgs^{-1}\big) \omega_G(ds)\\
\qquad{} = \frac{w(\hat{H}(g))}{w(\hat{H})} \omega_{\hat{H}(g)}(h )
\cdot \operatorname{ch}(\tau)(g) \int_G \operatorname{ch}(\tau_1)\big(sgs^{-1}\big) 1_{G(\chi)}\big(sgs^{-1}\big) \omega_G(ds)\\
\qquad{} = \big({\operatorname{ind}}_G^{H\rtimes_{\alpha}G} \big({\operatorname{ch}}(\tau) \big({\operatorname{res}}_G^{H\rtimes_{\alpha}G}\operatorname{ch}(\pi)\big)\big)\big)(h,g).
\end{gather*}
Hence we obtain the admissibility condition~(1).

(2) For $\tau \in \hat{G}$, applying Proposition~\ref{Proposition4.6},
\begin{gather*}
\operatorname{res}_G^{H\rtimes_{\alpha}G}\big({\operatorname{ind}}_G^{H\rtimes_{\alpha}G}\operatorname{ch}(\tau)\big)(g)= \frac{w(\hat{H}(g))}{w(\hat{H})} \cdot \operatorname{ch}(\tau)(g)= \frac{w(\hat{H}(g))}{w(\hat{H})} \cdot \operatorname{ch}(\tau_0)(g)\cdot \operatorname{ch}(\tau)(g)\\
\hphantom{\operatorname{res}_G^{H\rtimes_{\alpha}G}\big({\operatorname{ind}}_G^{H\rtimes_{\alpha}G}\operatorname{ch}(\tau)\big)(g)}{}
= \operatorname{res}_G^{H\rtimes_{\alpha}G}\big({\operatorname{ind}}_G^{H\rtimes_{\alpha}G}\tau_0\big)(g) \cdot \operatorname{ch}(\tau)(g).
\end{gather*}
Hence we see the admissibility condition (2).

By Theorem~\ref{Theorem3.8} we see that $\mathcal{K}\big(\widehat{H\rtimes_{\alpha}G} \cup \hat{G},\mathbb{Z}_q(2)\big)$ is a discrete commutative hypergroup.
\end{proof}

\section[Examples of $\mathcal{K}\big(\hat{H} \cup \widehat{H_0}, \mathbb{Z}_q(2)\big)$]{Examples of $\boldsymbol{\mathcal{K}\big(\hat{H} \cup \widehat{H_0}, \mathbb{Z}_q(2)\big)}$}\label{section5}

In this section we illustrate some special case of
\begin{gather*}\mathcal{K}\big(\hat{H} \cup \widehat{H_0}, \mathbb{Z}_q(2)\big)
= \big\{(\operatorname{ch}(\pi_i),\circ), (\operatorname{ch}(\tau_j), \bullet) \colon \pi_i \in \hat{H}, \tau_j \in \widehat{H_0}\big\}.
\end{gather*}
We denote $(\operatorname{ch}(\pi_i),\circ)$ for $\pi_i \in \hat{H}$ by $\gamma_i$ and $(\operatorname{ch}(\tau_j), \bullet)$ for $\tau_j \in \widehat{H_0}$ by $\rho_j$ hereafter.

\begin{Example}\label{Example5.1} Let $H = \mathbb{Z}_p(2) = \{h_0, h_1\}$ and $H_0 = \{h_0\}$ where $h_0$ is unit of $H$. In this case $\hat{H} = \{\pi_0, \pi_1\}$ and $\widehat{H_0} = \{\tau_0\}$. The Frobenius diagram (we refer to~\cite{HKTY1}) is
$$
\includegraphics[scale=1]{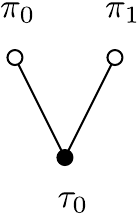}
$$
The structure equations of $\mathcal{K}\big(\widehat{\mathbb{Z}_p(2)} \cup \widehat{ \{h_0\} }, \mathbb{Z}_q(2)\big) = \{\gamma_0, \gamma_1, \rho_0\}$ are
\begin{gather*}
\gamma_1 \gamma_1 = p \gamma_0 + (1-p)\gamma_1,\qquad \rho_0 \rho_0 = \frac{q}{2}\gamma_0 + \frac{q}{2}\gamma_1 + (1-q) \rho_0,\qquad \gamma_1 \rho_0 = \rho_0.
\end{gather*}
We note that $\mathcal{K}\big(\widehat{\mathbb{Z}_p(2)} \cup \widehat{\{h_0\}}, \mathbb{Z}_q(2)\big) \cong \mathbb{Z}_p(2) \vee \mathbb{Z}_q(2)$ and $(p,q)$-deformations of the hypergroup associated with Dynkin diagram $A_3$ refer to Sunder--Wildberger~\cite{SW}.
\end{Example}

\begin{Example}\label{Example5.2} Let $H = \mathbb{Z}_p(3) = \{h_0, h_1, h_2\}$ and $H = \{ h_0 \}$ In this case $\hat{H} = \{\pi_0, \pi_1, \pi_2 \}$ and
 $\widehat{H_0} = \{\tau_0\}$. The Frobenius diagram is
$$
\includegraphics[scale=1]{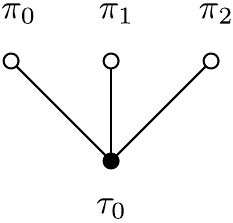}
$$
The structure equations of $\mathcal{K}\big(\widehat{\mathbb{Z}_p(3)} \cup \widehat{\{h_0\}}, \mathbb{Z}_q(2)\big) = \{\gamma_0, \gamma_1, \gamma_2, \rho_0\}$
are
\begin{gather*}
\gamma_1 \gamma_1 = \frac{1-p}{2}\gamma_1 + \frac{1+p}{2}\gamma_2,\qquad \gamma_2 \gamma_2 = \frac{1+p}{2}\gamma_1 + \frac{1-p}{2}\gamma_2,\\
\gamma_1 \gamma_2 = p\gamma_0 + \frac{1-p}{2} \gamma_1 + \frac{1-p}{2}\gamma_2,\qquad \rho_0 \rho_0 = \frac{q}{3}\gamma_0 + \frac{q}{3}\gamma_1 + \frac{q}{3}\gamma_2 + (1-q)\rho_0,\\
\gamma_1 \rho_0 = \gamma_2 \rho_0 = \rho_0.
\end{gather*}

We note that $\mathcal{K}\big(\widehat{\mathbb{Z}_p(3)} \cup \widehat{\{h_0\}}, \mathbb{Z}_q(2)\big) \cong \mathbb{Z}_p(3) \vee \mathbb{Z}_q(2)$ and these hypergroups are $(p,q)$-deformations of the hypergroup associated with Dynkin diagram $D_4$ constructed by Sunder--Wildberger~\cite{SW}.
\end{Example}

\begin{Example}\label{Example5.3} Let $H = \mathbb{Z}_{(p,r)}(4) = \{h_0, h_1, h_2, h_3\}$ and $H_0 = \{h_0\}$. In this case $\hat{H} = \{\pi_0, \pi_1$, $\pi_2, \pi_3\}$ and $\widehat{H_0} = \{\tau_0\}$. The Frobenius diagram is
$$
\includegraphics[scale=1]{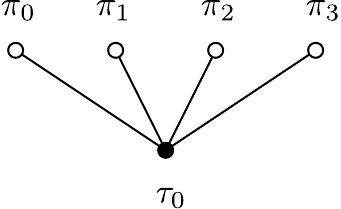}
$$
The structure equations of $\mathcal{K}\big(\widehat{\mathbb{Z}_{(p,r)}(4)} \cup \widehat{\{h_0\}}, \mathbb{Z}_q(2)\big)= \{\gamma_0, \gamma_1, \gamma_2, \gamma_3, \rho_0\}$ are
\begin{gather*}
\gamma_1 \gamma_1 = \gamma_3 \gamma_3 = \frac{1-p}{2}\gamma_1 + p\gamma_2 + \frac{1-p}{2}\gamma_3,\qquad \gamma_2 \gamma_2 = r\gamma_0 + (1-r)\gamma_2,\\
\gamma_1 \gamma_2 = \frac{1-r}{2}\gamma_1 + \frac{1+r}{2}\gamma_3,\qquad \gamma_1 \gamma_3 = \frac{2pr}{1 + r}\gamma_0 +
\frac{1-p}{2}\gamma_1 + \frac{p-pr}{1+r}\gamma_2 + \frac{1-p}{2}\gamma_3,\\
\gamma_2 \gamma_3 = \frac{1+r}{2}\gamma_1 + \frac{1-r}{2}\gamma_3,\qquad
\rho_0 \rho_0 = \frac{q}{4}\gamma_0 + \frac{q}{4}\gamma_1 + \frac{q}{4}\gamma_2 + \frac{q}{4}\gamma_3 + (1-q) \rho_0,\\
\gamma_1 \rho_0 = \gamma_2 \rho_0 = \gamma_3 \rho_0 = \rho_0.
\end{gather*}
We note that $\mathcal{K}\big(\widehat{\mathbb{Z}_{(p,r)}(4)} \cup \widehat{\{h_0\}}, \mathbb{Z}_q(2)\big) \cong \mathbb{Z}_{(p,r)}(4) \vee \mathbb{Z}_q(2)$.
\end{Example}

\begin{Example}\label{Example5.4}Let $H = \mathbb{Z}_{(p,r)}(4) = \{h_0, h_1, h_2, h_3\}$ and $H_0 = \mathbb{Z}_p(2) = \{h_0, h_2\}$. In this case
$\hat{H} = \{\pi_0, \pi_1, \pi_2, \pi_3\}$ and $\widehat{H_0} = \{\tau_0, \tau_1\}$. The Frobenius diagram is
$$
\includegraphics[scale=1]{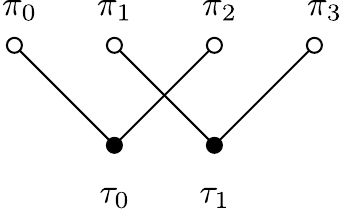}
$$
The structure equations of $\mathcal{K}\big(\widehat{\mathbb{Z}_{(p,r)}(4)} \cup \widehat{\mathbb{Z}_p(2)}, \mathbb{Z}_q(2)\big)= \{\gamma_0, \gamma_1, \gamma_2, \gamma_3, \rho_0, \rho_1\}$ are
\begin{gather*}
\gamma_1 \gamma_1 = \gamma_3 \gamma_3 = \frac{1-p}{2}\gamma_1 + p\gamma_2 + \frac{1-p}{2}\gamma_3,\qquad
\gamma_2 \gamma_2 = r\gamma_0 + (1-r)\gamma_2,\\
\gamma_1 \gamma_2 = \frac{1-r}{2}\gamma_1 + \frac{1+r}{2}\gamma_3,\qquad \gamma_1 \gamma_3 = \frac{2pr}{1 + r}\gamma_0 +
\frac{1-p}{2}\gamma_1 + \frac{p-pr}{1+r}\gamma_2 + \frac{1-p}{2}\gamma_3,\\
\gamma_2 \gamma_3 = \frac{1+r}{2}\gamma_1 + \frac{1-r}{2}\gamma_3,\qquad
\rho_0 \rho_0 = \rho_1 \rho_1 = \frac{q}{2}\gamma_0 + \frac{q}{2}\gamma_2 + (1-q)\rho_0,\\
\rho_0 \rho_1 = \frac{q}{2}\gamma_1 + \frac{q}{2}\gamma_3 + (1-q)\rho_1,\qquad
\gamma_1 \rho_0 = \gamma_3 \rho_0 = \gamma_2 \rho_1 = \rho_1,\\
\gamma_1 \rho_1 = \gamma_3 \rho_1 = \gamma_2 \rho_0 = \rho_0.
\end{gather*}
\end{Example}

\begin{Example}\label{Example5.5} Let $H = S_p(3) = \mathbb{Z}_p(3) \rtimes_\alpha \mathbb{Z}_2$ and $H_0 = \mathbb{Z}_2$. In this case $\hat{H} = \{\pi_0, \pi_1, \pi_2\}$ $(\dim \pi_2 = 2)$ and $\widehat{H_0} = \{\tau_0, \tau_1\}$. The Frobenius diagram is
$$
\includegraphics[scale=1]{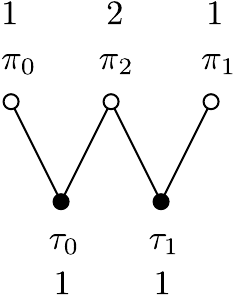}
$$
The structure equations of $\mathcal{K}\big(\widehat{S_p(3)} \cup \widehat{\mathbb{Z}_2}, \mathbb{Z}_q(2)\big) = \{\gamma_0, \gamma_1, \gamma_2, \rho_0, \rho_1\}$ are
\begin{gather*}
\gamma_1\gamma_1 = \gamma_0,\qquad \gamma_2 \gamma_2 = \frac{p}{4}\gamma_0 + \frac{p}{4}\gamma_1 + \left(1- \frac{p}{2} \right)\gamma_2,\qquad
\gamma_1 \gamma_2 = \gamma_2,\\
\rho_0\rho_0 = \rho_1 \rho_1 = \frac{q}{3}\gamma_0 + \frac{2q}{3}\gamma_2 + (1-q)\rho_0,\qquad \rho_0 \rho_1 = \frac{q}{3}\gamma_1 + \frac{2q}{3}\gamma_2 + (1-q)\rho_1,\\
\gamma_1 \rho_0 = \rho_1,\qquad \gamma_1 \rho_1 = \rho_0,\qquad \gamma_2 \rho_0 = \gamma_2 \rho_1 = \frac{1}{2}\rho_0 + \frac{1}{2}\rho_1.
\end{gather*}

We note that these hypergroups are $(p,q)$-deformations of the hypergroup associated with Dynkin diagram $A_5$ refer to Sunder--Wildberger \cite{SW}.
\end{Example}

\begin{Example}\label{Example5.6}Let $H = D_{(p,r)}(4) = \mathbb{Z}_{(p,r)}(4) \rtimes_\alpha \mathbb{Z}_2$ and $H_0 = \mathbb{Z}_2$. In this case $\hat{H} = \{\pi_0, \pi_1, \pi_2, \pi_3$, $\pi_4\}$ $(\dim \pi_4 = 2)$ and $\widehat{H_0} = \{\tau_0, \tau_1\}$. The Frobenius diagram is
$$
\includegraphics[scale=1]{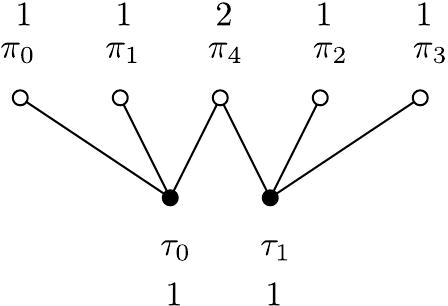}
$$
The structure equations of $\mathcal{K}\big(\widehat{D_{(p,r)}(4)} \cup \widehat{\mathbb{Z}_2}, \mathbb{Z}_q(2)\big) = \{\gamma_0, \gamma_1, \gamma_2, \gamma_3, \gamma_4, \rho_0, \rho_1\}$ are
\begin{gather*}
\gamma_1 \gamma_1 = \gamma_0,\qquad \gamma_2 \gamma_2 = \gamma_3 \gamma_3 = r\gamma_0 + (1-r)\gamma_2,\\
\gamma_4 \gamma_4 = \frac{pr}{2(1+r)}\gamma_0 + \frac{pr}{2(1+r)}\gamma_1 + \frac{p}{2(1+r)}\gamma_2 + \frac{p}{2(1+r)}\gamma_3 + (1-p)\gamma_4,\\
\gamma_1 \gamma_2 = \gamma_3,\qquad \gamma_1 \gamma_3 = \gamma_2,\qquad
\gamma_1 \gamma_4 = \gamma_2 \gamma_4 = \gamma_3 \gamma_4 = \gamma_4, \qquad \gamma_2 \gamma_3 = r\gamma_1 + (1-r)\gamma_3,\\
\rho_0 \rho_0 = \rho_1 \rho_1 = \frac{q}{4}\gamma_0 +\frac{q}{4}\gamma_1 + \frac{q}{2}\gamma_4 + (1-q)\rho_0,\qquad
\rho_0 \rho_1 = \frac{q}{4}\gamma_2 + \frac{q}{4}\gamma_3 + \frac{q}{2}\gamma_4 + (1-q)\rho_1,\\
\gamma_1 \rho_0 = \gamma_2 \rho_1 = \gamma_3 \rho_1 = \rho_0, \qquad \gamma_1 \rho_1 = \gamma_2 \rho_0 = \gamma_3 \rho_0 = \rho_1,\qquad
\gamma_4 \rho_0 = \gamma_4 \rho_1 = \frac{1}{2}\rho_0 + \frac{1}{2}\rho_1.
\end{gather*}
\end{Example}

\begin{Example}\label{Example5.7} Let $H = A_4 = (\mathbb{Z}_2 \times \mathbb{Z}_2) \rtimes_\alpha \mathbb{Z}_3$ and $H_0 = \mathbb{Z}_3$. In this case $\hat{H} = \{\pi_0, \pi_1, \pi_2, \pi_3 \}$ $(\dim \pi_3 = 3)$ and $\widehat{H_0} = \{\tau_0, \tau_1, \tau_2\}$. The Frobenius diagram is
$$
\includegraphics[scale=1]{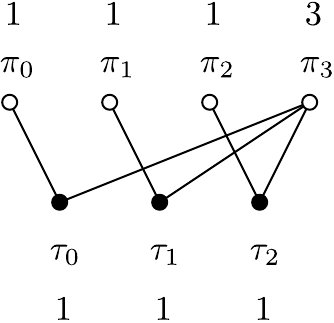}
$$
The structure equations of $\mathcal{K}\big(\widehat{A_4} \cup \widehat{\mathbb{Z}_3}, \mathbb{Z}_q(2)\big) = \{\gamma_0, \gamma_1, \gamma_2, \gamma_3, \gamma_4, \rho_0, \rho_1, \rho_2\}$ are
\begin{gather*}
\gamma_1 \gamma_1 = \gamma_2,\qquad \gamma_2 \gamma_2 = \gamma_1,\qquad \gamma_3 \gamma_3
= \frac{1}{9}\gamma_0 + \frac{1}{9}\gamma_1 + \frac{1}{9}\gamma_2 + \frac{2}{3}\gamma_3,\qquad
\gamma_1 \gamma_2 = \gamma_0,\\
\gamma_1 \gamma_3 = \gamma_3,\qquad \gamma_2 \gamma_3 = \gamma_3,\qquad \rho_0 \rho_0 = \rho_1 \rho_2 = \frac{q}{4}\gamma_0 + \frac{3q}{4}\gamma_3 + (1-q)\rho_0,\\
\rho_1 \rho_1 = \rho_0 \rho_2 = \frac{q}{4}\gamma_2 + \frac{3q}{4}\gamma_3 + (1-q)\rho_2,\qquad
\rho_2 \rho_2 = \rho_0 \rho_1 = \frac{q}{4}\gamma_1 + \frac{3q}{4}\gamma_3 + (1-q)\rho_1,\\
\gamma_1 \rho_0 = \rho_1,\qquad \gamma_2 \rho_0 = \gamma_1 \rho_1 = \rho_2,\qquad \gamma_2 \rho_1 = \rho_0,\\
\gamma_3 \rho_0 = \gamma_3 \rho_1 = \gamma_3 \rho_2 = \frac{1}{3}\rho_0 + \frac{1}{3}\rho_1 + \frac{1}{3}\rho_2.
\end{gather*}
\end{Example}

\begin{Example}\label{Example5.8} Let
$H = S_4 = A_4 \rtimes_\alpha \mathbb{Z}_2$ and $H_0 = \mathbb{Z}_2$. In this case $\hat{H} = \{\pi_0, \pi_1, \pi_2, \pi_3, \pi_4 \}$ and $\widehat{H_0} = \{\tau_0, \tau_1\}$. The Frobenius diagram is
$$
\includegraphics[scale=1]{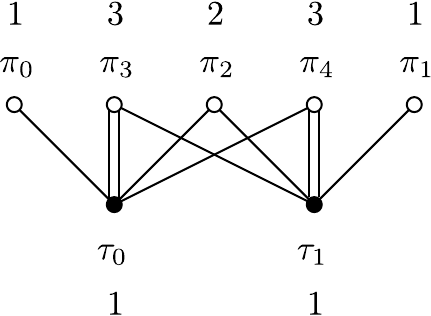}
$$
The structure equations of $\mathcal{K}\big(\widehat{S_4} \cup \widehat{\mathbb{Z}_2}, \mathbb{Z}_q(2)\big) = \{\gamma_0, \gamma_1, \gamma_2, \gamma_3, \gamma_4, \rho_0, \rho_1\}$ are
\begin{gather*}
\gamma_1 \gamma_1 = \gamma_0,\qquad
\gamma_2 \gamma_2 = \frac{1}{4}\gamma_0 + \frac{1}{4}\gamma_1 + \frac{1}{2}\gamma_2,\\
\gamma_3 \gamma_3 = \gamma_4 \gamma_4
= \frac{1}{9}\gamma_0 + \frac{2}{9}\gamma_2 + \frac{1}{3}\gamma_3 + \frac{1}{3}\gamma_4,\\
\gamma_1 \gamma_2 = \gamma_2,\qquad \gamma_1 \gamma_3 = \gamma_4,\qquad
\gamma_1 \gamma_4 = \gamma_3,\qquad
\gamma_2 \gamma_3 = \gamma_2 \gamma_4
= \frac{1}{2}\gamma_3 + \frac{1}{2}\gamma_4,\\
\gamma_3 \gamma_4
= \frac{1}{9}\gamma_1 + \frac{2}{9}\gamma_2 + \frac{1}{3}\gamma_3 + \frac{1}{3}\gamma_4,\\
\rho_0 \rho_0 = \rho_1 \rho_1 =
\frac{q}{12}\gamma_0 + \frac{q}{6}\gamma_2 + \frac{q}{2}\gamma_3 +
\frac{q}{4}\gamma_4 + (1-q)\rho_0,\\
\rho_0 \rho_1 = \rho_1 \rho_0 =
\frac{q}{12}\gamma_1 + \frac{q}{6}\gamma_2 + \frac{q}{4}\gamma_3 +
\frac{q}{2}\gamma_4 + (1-q)\rho_1,\\
\gamma_0 \rho_0 = \gamma_1 \rho_1 = \rho_0,\qquad \gamma_0 \rho_1 = \gamma_1 \rho_0 = \rho_1,\qquad
\gamma_2 \rho_0 = \gamma_2 \rho_1 = \frac{1}{2}\rho_0 + \frac{1}{2}\rho_1,\\
\gamma_3 \rho_0 = \gamma_4 \rho_1 = \frac{2}{3}\rho_0 + \frac{1}{3}\rho_1,\qquad
\gamma_4 \rho_0 = \gamma_3 \rho_1 = \frac{1}{3}\rho_0 + \frac{2}{3}\rho_1.
\end{gather*}
\end{Example}

\begin{Example}\label{Example5.9} Let $H = S_4 $ and $H_0 = S_3$. In this case $\hat{H} = \{\pi_0, \pi_1, \pi_2, \pi_3, \pi_4 \}$ and $\widehat{H_0} = \{\tau_0, \tau_1, \tau_2\}$. The Frobenius diagram is
$$
\includegraphics[scale=1]{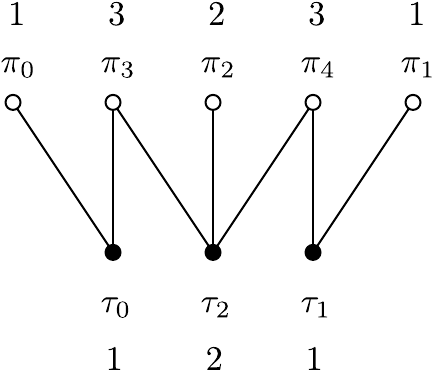}
$$
The structure equations of $\mathcal{K}\big(\widehat{S_4} \cup \widehat{S_3}, \mathbb{Z}_q(2)\big) = \{\gamma_0, \gamma_1, \gamma_2, \gamma_3, \gamma_4, \rho_0, \rho_1, \rho_2\}$ are
\begin{gather*}
\gamma_1 \gamma_1 = \gamma_0,\qquad \gamma_2 \gamma_2 = \frac{1}{4}\gamma_0 + \frac{1}{4}\gamma_1 + \frac{1}{2}\gamma_2,\qquad
\gamma_3 \gamma_3 = \gamma_4 \gamma_4 = \frac{1}{9}\gamma_0 + \frac{2}{9}\gamma_2 + \frac{1}{3}\gamma_3 + \frac{1}{3}\gamma_4,\\
\gamma_1 \gamma_2 = \gamma_2,\qquad \gamma_1 \gamma_3 = \gamma_4,\qquad \gamma_1 \gamma_4 = \gamma_3,\qquad
\gamma_2 \gamma_3 = \gamma_2 \gamma_4 = \frac{1}{2}\gamma_3 + \frac{1}{2}\gamma_4,\\
\gamma_3 \gamma_4 = \frac{1}{9}\gamma_1 + \frac{2}{9}\gamma_2 + \frac{1}{3}\gamma_3 + \frac{1}{3}\gamma_4,\qquad
\rho_0 \rho_0 = \rho_1 \rho_1 = \frac{q}{4}\gamma_0 + \frac{3q}{4}\gamma_3 + (1-q)\rho_0,\\
 \rho_2 \rho_2 = \frac{q}{16}\gamma_0 + \frac{q}{16}\gamma_1 + \frac{q}{8}\gamma_2 + \frac{3q}{8}\gamma_3 + \frac{3q}{8}\gamma_4 +
\frac{1-q}{4}\rho_0 + \frac{1-q}{4}\rho_1 + \frac{1-q}{2}\rho_2,\\
\rho_1 \rho_2 = \frac{q}{4}\gamma_2 + \frac{3q}{8}\gamma_3 + \frac{3q}{8}\gamma_4 + (1-q)\rho_2,\qquad
\gamma_0 \rho_0 = \gamma_1 \rho_1 = \rho_0,\qquad \gamma_0 \rho_1 = \gamma_1 \rho_0 = \rho_1,\\
\gamma_0 \rho_2 = \gamma_1 \rho_2 = \gamma_2 \rho_0 = \gamma_2 \rho_1 = \rho_2,\qquad
\gamma_3 \rho_0 = \gamma_4 \rho_1 = \frac{1}{3}\rho_0 + \frac{2}{3}\rho_2,\\
\gamma_4 \rho_0 = \gamma_3 \rho_1 = \frac{1}{3}\rho_1 + \frac{2}{3}\rho_2,\qquad
\gamma_2 \rho_2 = \frac{1}{4}\rho_0 + \frac{1}{4}\rho_1 + \frac{1}{2}\rho_2,\\
\gamma_3 \rho_2 = \gamma_4 \rho_0 = \frac{1}{6}\rho_0 + \frac{1}{6}\rho_1 + \frac{2}{3}\rho_2.
\end{gather*}
\end{Example}

\begin{Example}\label{Example5.10} Let $H = \mathbb{Z}_p(3) \rtimes_\alpha (\mathbb{Z}_2 \times \mathbb{T}) \cong S_p(3) \times \mathbb{T}$ and
$H_0 = \mathbb{Z}_2 \times \mathbb{T}$. In this case $\hat{H} = \widehat{S_p(3) \times \mathbb{T}} = \widehat{S_p(3)} \times \hat{\mathbb{T}} = \widehat{S_p(3)} \times \mathbb{Z}$ and $\widehat{H_0} = \widehat{\mathbb{Z}_2} \times \hat{\mathbb{T}} = \mathbb{Z}_2 \times \mathbb{Z}$. The Frobenius diagram is
\begin{center}
\includegraphics[scale=1]{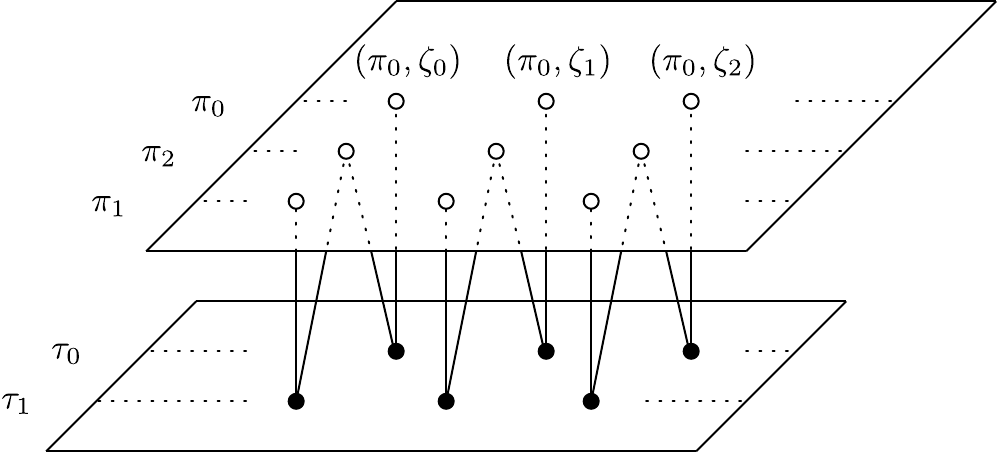}
\end{center}
We denote $(\operatorname{ch}((\pi_i, \zeta_j)), \circ)$ for $(\pi_i, \zeta_j) \in \hat{H}$ by $(\gamma_i, \xi_j)$ and $(\operatorname{ch}((\tau_k, \zeta_\ell)), \bullet)$ for $(\tau_k, \zeta_\ell) \in \widehat{H_0}$ by $(\rho_k, \xi_\ell)$. The structure equations of
\begin{gather*}
\mathcal{K}\big(\hat{H} \cup \widehat{H_0}, \mathbb{Z}_q(2)\big)
= \{(\gamma_i, \xi_j), (\rho_k, \xi_\ell) \colon i = 0,1,2, \ k = 0,1, \ j,\ell \in \mathbb{Z}\}
\end{gather*}
are
\begin{gather*}
(\gamma_1, \xi_j) (\gamma_1, \xi_{\ell}) = (\gamma_0, \xi_{j + \ell}),\qquad
(\gamma_1, \xi_j) (\gamma_2, \xi_{\ell}) = (\gamma_2, \xi_{j + \ell}),\\
(\gamma_2, \xi_j) (\gamma_2, \xi_{\ell}) = \frac{p}{4}(\gamma_0, \xi_{j + \ell}) + \frac{p}{4}(\gamma_1, \xi_{j + \ell}) +
\left(1-\frac{p}{2}\right)(\gamma_2, \xi_{j + \ell}),\\
(\rho_0, \xi_j) (\rho_0, \xi_{\ell}) = (\rho_1, \xi_j) (\rho_1, \xi_{\ell}) =
\frac{q}{3}(\gamma_0, \xi_{j + \ell}) + \frac{2q}{3}(\gamma_2, \xi_{j + \ell})
+ (1-q)(\rho_0, \xi_{j + \ell}),\\
(\rho_0, \xi_j) (\rho_1, \xi_{\ell}) = (\rho_1, \xi_\ell) (\rho_0, \xi_j) =
\frac{q}{3}(\gamma_1, \xi_{j + \ell}) + \frac{2q}{3}(\gamma_2, \xi_{j + \ell})
+ (1-q)(\rho_1, \xi_{j + \ell}),\\
(\gamma_0, \xi_j) (\rho_0, \xi_\ell) = (\gamma_1, \xi_j) (\rho_1, \xi_\ell) = (\rho_0, \xi_{j + \ell}),\\
(\gamma_0, \xi_j) (\rho_1, \xi_\ell) = (\gamma_1, \xi_j) (\rho_0, \xi_\ell) = (\rho_1, \xi_{j + \ell}),\\
(\gamma_2, \xi_j) (\rho_0, \xi_\ell) = (\gamma_2, \xi_j) (\rho_1, \xi_\ell) = \frac{1}{2}(\rho_0, \xi_{j + \ell}) + \frac{1}{2}(\rho_1, \xi_{j + \ell}).
\end{gather*}
\end{Example}

\pdfbookmark[1]{References}{ref}
\LastPageEnding

\end{document}